\documentclass[a4paper,oneside,10pt]{amsart}
\usepackage[utf8]{inputenc}

\usepackage{amsmath,amsfonts,amssymb,amsthm,amscd,latexsym,mathrsfs, mathabx, amsbsy, bbm, multicol}

\usepackage{mathtools}
\usepackage[usenames, dvipsnames]{color}

\usepackage{enumitem}
\usepackage{fancyhdr}
\usepackage{stmaryrd,xypic}
\usepackage[hidelinks]{hyperref}

\usepackage{verbatim}
\usepackage{tikz-cd}
\theoremstyle{definition}

\newtheorem{definition}{Definition}[section]
\newtheorem{proposition}[definition]{Proposition}
\newtheorem{lemma}[definition]{Lemma}
\newtheorem{theorem}[definition]{Theorem}
\newtheorem{corollary}[definition]{Corollary}

\newtheorem{remark}[definition]{Remark}

\newtheorem{criteria}[definition]{Criteria}

\newcommand{\bbA}{\mathbb{A}}

\newcommand{\bbG}{\mathbb{G}}
\newcommand{\bbP}{\mathbb{P}}
\newcommand{\bbN}{\mathbb{N}}
\newcommand{\bbZ}{\mathbb{Z}}
\newcommand{\bbR}{\mathbb{R}}

\newcommand{\bbT}{\mathbb{T}}

\newcommand{\bbL}{\mathbb{L}}
\newcommand{\bbI}{\mathbb{I}}
\newcommand{\bbS}{\mathbb{S}}
\newcommand{\bbU}{\mathbb{U}}

\newcommand{\KC}{\mathcal{C}}
\newcommand{\KG}{\mathcal{G}}

\newcommand{\SL}{\mathscr{L}}
\newcommand{\SX}{\mathscr{X}}

\newcommand{\SQ}{\mathscr{Q}}

\newcommand{\ft}{\mathfrak{t}}

\newcommand{\fS}{\mathfrak{S}}

\newcommand{\bfH}{\mathbf{H}}
\newcommand{\bfI}{\mathbf{I}}
\newcommand{\bfR}{\mathbf{R}}
\newcommand{\bfX}{\mathbf{X}}

\newcommand{\bfD}{\mathbf{D}}

\newcommand{\bfCh}{\mathbf{Ch}}

\newcommand{\bfDelta}{\mathbf{\Delta}}

\newcommand{\bfGamma}{\mathbf{\Gamma}}

\newcommand{\bfp}{\mathbf{p}}

\newcommand{\bff}{\mathbf{f}}
\newcommand{\bfu}{\mathbf{u}}

\newcommand{\bfe}{\mathbf{e}}

\DeclareMathOperator{\spec}{Spec}

\DeclareMathOperator{\an}{\mathrm{an}}

\DeclareMathOperator{\dd}{d^\prime d^{\prime\prime}}
\DeclareMathOperator{\MA}{\mathrm{MA}}
\DeclareMathOperator{\trop}{\mathrm{trop}}

\DeclareMathOperator{\Div}{\mathrm{div}}

\DeclareMathOperator{\Id}{\mathrm{Id}}

\DeclareMathOperator{\Pic}{\mathrm{Pic}}

\DeclareMathOperator{\Ch}{\mathrm{Ch}}
\DeclareMathOperator{\Sh}{\Check{S}}

\DeclareMathOperator{\dom}{\mathrm{dom}}

\DeclareMathOperator{\wt}{\mathrm{wt}}

\newcommand{\ndot}{\mathord{\cdot}}
\DeclarePairedDelimiter{\norm}{\lVert}{\rVert}
\DeclarePairedDelimiter{\abs}{\lvert}{\rvert}

\DeclarePairedDelimiter{\Card}{[\![}{]\!]}

\DeclarePairedDelimiter{\braket}{\langle}{\rangle}

\begin{document}
\title{Critical Fubini-Study metrics over non-archimedean field}
\author{Yanbo Fang}
\begin{abstract}
    Over a non-archimedean local place, the height of a projective variety with respect to a very ample line bundle equipped with a Fubini-Study metric is related to the naive height of its Chow form. Using a non-Archimedean Kempf-Ness criteria, we characterize Fubini-Study metrics that minimize the height under the special linear action in terms of their Monge-Ampère polytopes. This polytope can be constructed either as its non-Archimedean Bergman functional or as the weight polytope for the residual action on its Chow form; it is associated with a polymatroid.
\end{abstract}
\maketitle

\section{Introduction}

\subsection{Height of projective varieties over global or local fields}~

Analogous to the construction of heights for projective varieties\footnote{It is also called Faltings' height (for cycles), and we call it \emph{projective height} to avoid confusion with the canonical height of Abelian varieties; note that in \cite{BoGS} this terminology refers to another intersection $[\SX]\cdot c_{d+1}(\SQ)$ which we prefer to call \emph{dual projective height} in future work in accordance with the duality picture given in \textit{idbd.}.} over a global field, in Arakelov geometry over a non-archimedean local place, the height of a projective variety $X$ of pure dimension $d$ with very ample line bundle $L$ over a non-archimedean field $K$ was defined as a degree-like intersection number $[\SX]\cdot c_1(\SL)^{d+1}$ over a model $(\SX,\SL)$ of $(X,L)$ over the valuation ring; it is shown to be equal to a local Weil height of its Chow form $R_{\SX}$ \cite{Gub}, and can also be expressed via the intersection line bundle $\braket{\SL}^{\braket{d+1}}$ \cite{BE}. These constructions can be recasted in differentio-geometric terms in parallel with the situation over an Archimedean place where the local height is given by Green pairings of differential forms, thanks to the fact that an ample integral model can be equivalently described by its induced (non-Archimedean) Fubini-Study metric on the (Berkovich) analytification $L^{\an}$ over $X^{\an}$. Unlike the global case, besides the metric, the projective height over a local place also depends on the choice of $(d+1)$ global sections of $\SL$; this dependence disappear when local projective heights add up to the global one, thanks to the product formula (see \cite{Gub}\cite{CM}). 

Geometric invariant theory can be applied to Arakelov geometry over number fields (see for example \cite{Bur}\cite{Zha2}\cite{Gas}\cite{Che}\cite{Mac}). For a subvariety $X$ of $\bbP^n$, the natural action by $\bbS\bbL(n+1)$ moves the cycle and changes its projective height. The minima of projective heights of cycles in such an orbit were studied; the existence requires semi-stability of its Chow form, and lower bounds were obtained based on estimates of successive minima for arithmetic varieties \cite{Sou}\cite{Bos1}\cite{Zha2}. Thanks to the Kempf-Ness criteria over complex numbers and its non-Archimedean version, minima of local projective heights were known to be attained at cycles in the orbit whose Chow form has semi-stable reduction. Equivalently, one can fix the subvariety and look at the orbit of Fubini-Study metrics at each place under the action, and for a chosen place, call the Fubini-Study metrics on $L^{\an}$ that minimizes the corresponding local projective height under the $\bbS\bbL(n+1)$-action \emph{critical (F.-S) metrics}\cite{Zha1}.

\subsection{Critical Fubini-Study metrics over a local place}~

Over an Archimedean place, by directly calculating the Euler-Lagrange equation for the variation of heights, critical Fubini-Study metrics\footnote{Indeed over an Archimedean place, criticality can be defined among all semipositive continuous metrics, and critical ones are shown to be Fubini-Study ones; we don't know if this holds over a non-Archimedean place.} on $L^{\an}$ are characterized by the constancy of its Bergman distorsion function on $X^{\an}$\cite{Zha1}\cite{Bos2}. Such critical Fubini-Study metrics and their characterization in terms of Bergman functions are rediscovered as \emph{balanced metrics} in Kähler geometry and are exploited as approximations to a putative metric on $L^{\an}$ whose curvature form is positive which gives a constant scalar curvature Riemannian metric on $X^{\an}$ \cite{Luo}\cite{Don}. 

This article proposes a charcterization of critical Fubini-Study metrics in the non-Archimedean case. We give two approaches, both rely on the non-Archiemdean Kempf-Ness theorem and differentio-geometric techniques on Berkovich analytic spaces. The first follows the strategy of the Archimedean case by directly calculating the variation of projective height, using its equivalent expression as a metric on the intersection line bundle and a corresponding variational formula \cite{BE}. Another approach exploits the non-Archimedean Kempf-Ness criteria \cite{Bur}\cite{Mac} and focuses on the semi-stability of the reduction of Chow form, by expressing its weight polytope under the residual action using tropical intersection implemented by the calculus of $\delta$-forms \cite{CLD}\cite{GK}\cite{Mih}. 

Both leads to the same criteria, in terms of an object constructed purely from $(X,L;\phi_{\norm{\ndot}})$: upon choosing an orthonormal basis in $H^0(\bbP^n,O(1);\norm{\ndot})$, the analogue of Bergman distorsion function is a polytope in $\bbR^{n+1}/\pmb{1}\bbR$, denoted by $P_{\MA}(X,L;\phi_{\norm{\ndot}})$ as a finite Minkowski sum of simplices $\sum \mu_a\triangle_{I(a)}$ indexed by the support and dilated by the multiplicity, of the non-Archimedean Monge-Ampère measure $c_1(L,\phi_{\norm{\ndot}})^{\wedge d}$. The characterization for critical Fubini-Study metric is
\begin{theorem}(\ref{T: critical metric}, \ref{T': critical metric})
Assume that $(X,L)$ is Chow stable. A Fubini-Study metric associated with a strict Cartesian norm $\norm{\ndot}$ is critical if and only if $\pmb{0}\in P_{\MA}(X,L;\phi_{\norm{\ndot}})$. 
\end{theorem}

\subsection{Organization of article} In this article, we denote by
\begin{itemize}
    \item $K$ a non-archimedean field with a non-trivial absolute value $\abs{\ndot}$, with the valuation ring $K^{\circ}$, the residue field $\widetilde{K}$, and  $\overline{K}$ the completion of its algebraic closure; through out this article it is assumed that $K=\overline{K}$ (except in Criteria \ref{C: KN criteria})
    \item $E$ a $K$-vector space of rank $(n+1)$ and $\norm{\ndot}$ a strict Cartesian ultrametric norm on $E$;
    \item $X$ a closed irreducible subvariety of $\bbP(E)$, of dimension $d$ and degree $\delta$;
    \item $L$ the tautological $O(1)$ sheaf of rank $1$, $\overline{L}$ with the tautological metrization by $\norm{\ndot}$: for $L$ it is the Fubini-Study metric $\phi:=\phi_{\norm{\ndot}}$; $\SL$, $\SX$ corresponding model objects over $K^{\circ}$ for the model structure induced by $\norm{\cdot}$.
    \item $\Card{n}$ the set $\{0,1,\dots, n\}$ for any $n\in\bbN$
\end{itemize}
In \S2 we first recall standard constructions over a non-Archimedean field about Chow forms and intersection line bundles as as well as metrics on them; then we show the criticality criteria using a standard calculation of the variation of Chow norms in terms of the variation of an intersection metric. In \S3 we first recall tropical intersection theory and their incarnation into calculus of $\delta$-forms, then translate the non-Archimedean Kempf-Ness criteria for Chow forms by convex analysis to the criticality criteria.

\section{Analytic aspects}
\subsection{Fubini-Study metric and Monge-Ampère measure}~

For a projective variety $X$ with a very ample line bundle $L$, denote by $E$ the $K$-vector space $H^0(X,L)$ of dimension $(n+1)$ and by $\check{E}$ its dual vector space. Embed $X$ into $\bbP(E)$, and let $\bbP(\check{E})$ be its dual projective space. Let $\underline{x}$ be a homogeneous coordinates and $\underline{\xi}$ be its dual homogeneous coordinates, and $H$ and $\check{H}$ be the cycle classes of hyperplanes. Let $\bbI$ be the incidence divisor on $\bbP(E)\times \bbP(\check{E})$.

Let $\norm{\ndot}$ be a strict $K$-Cartesian ultrametric norm on $E$ with dual norm $\norm{\ndot}^{\vee}$ on $\check{E}$, it induces a Fubini-Study metric $\phi_{\norm{\ndot}}$ on $L^{\an}$; it coincides with the model metric associated with the integral structure on $\bbP^n$ and $O(1)$ induced by any orthonormal basis $\underline{e}$ of $(E,\norm{\ndot})$; in formula $\abs{s(x)}_{\phi_{\norm{\ndot}}}=\abs{s(\hat{x})}/\norm{\hat{x}}^{\vee}$ for any $s\in E$ and $x\in(\bbP(E))^{\an}$ with an affine representative $\hat{x}\in (\bbA^{n+1}(E))^{\an}$, explicitly $\hat{x}\in \check{E}_{K'}$ for some valued field extension $K'/K$. Only such strict diagonalizable Fubini-Study metrics will be concerned. 

Fix an orthonormal basis $\underline{e}$ of $E$, for any $\hat{x}\in \check{E}_{K'}$ written as $\sum_{i=0}^d a_i e^{\vee}_i$, denote by $I_{\underline{e}}(\hat{x})$ the subset of $\Card{n}$ of indices such that $\abs{a_i}=\norm{\hat{x}}^{\vee}$, this assignment descends to $x\in \bbP(E)^{\an}$ giving a subset function $I_{\underline{e}}(x)$.

\subsection{Chow form and Chow norm}

\subsubsection{Chow hypersurface}~

Denote the diagonal embedding of $E$ (or $\check{E}$, or $\bbP(E)$, or $\bbP(\check{E})$) into its $(d+1)$-fold product by $\Delta$. Denote by $p$ and $q$ the projections from $\bbP(E)^{d+1}\times \bbP(\check{E})^{d+1}$ to the two factors. For $0\leq i\leq d$, denote by the projections to two factors and the incidence divisor $\bbI_i$ on the $i$-th copy $\bbP(E)\times \bbP(\check{E})$. 

With respect to the embedding into $\bbP(E)$ by $O(1)$, the Chow variety $\Ch(X)$ of $X$ is the hypersurface in $\bbP(\check{E})^{d+1}$ defined by the integral transform $q_*\{p^*\Delta_* X\cdot\prod_{i=0}^{d}\bbI_i\}$. The pull-back $\Delta^*\Ch(X)$ as a Cartier divisor is a hypersurface in $\bbP(\check{E})$.

The Chow form $R_X$ of $X$ is the multi-homogeneous polynomial of multi-degree $(\delta,\dots,\delta)$ in $\Xi:=(\underline{\xi}^0,\dots,\underline{\xi}^d)\in \check{E}^{d+1}$, unique up to a multiplicative scalar, that defines the Chow hypersurface; it is a section of $\boxtimes_{i=0}^d O(1)_{\check{\bbP}^n_i}^{\delta}$. It also determines a vector $R_X$ (up to scalar dilation) in $(\fS^{\delta} \check{E})^{\otimes d+1}$, denote by $\{R_X\}$ this one-dimensional $K$-vector space. The pull-back $\Delta^*R_X$ that defines $\Delta^*\Ch(X)$ is a homogeneous polynomial of degree $\delta(d+1)$ in $\underline{\xi}\in\check{E}$.

These constructions are compatible with respect to base field change $K'/K$.

\subsubsection{Non-Archimedean Chow metric}(see \cite{Gub}\cite{CM})~

Let $\norm{\ndot}$ be a strict $K$-Cartesian norm on $E$, its dual norm on the dual $\check{E}$ is denoted by $\norm{\cdot}^{\vee}$. It induces a norm on $\otimes^{\delta} \check{E}$ by $\delta$-fold ($\epsilon$-)tensor product of $\norm{\cdot}^{\vee}$ hence a quotient norm denoted by $\norm{\ndot}_{\fS_\delta^{\vee}}$ on $\fS^{\delta}\check{E}$. 
The induced norm on $(\fS^{\delta} \check{E})^{\otimes d+1}$, denoted by $\norm{\ndot}_{\fS_{\delta,d+1}^{\vee}}$, is the Chow norm induced by $\norm{\ndot}$. Explicitly, let $\underline{e}$ be an orthonormal basis of $(E,\norm{\ndot})$ and $\underline{\check{e}}$ be the dual orthonormal basis of $(\check{E},\norm{\ndot}^{\vee})$, write for an element $R$ of the above space as $R=\sum_{\abs{I}=\delta} a_I \underline{\check{e}}^I$, then $\norm{R}_{\fS_{\delta,d+1}^{\vee}}=\max_{\abs{I}=\delta}\{\abs{a_I}\}$; thus $\{R_X\}$ as a line bundle over $\spec K$ is equipped with this norm. 

Originally, the Chow norm was defined by a local arithmetic intersection number of metrized divisors (see \cite[Proposition 9.7]{Gub} or \cite[Theorem 3.9.7]{CM})
\[\log~(\norm{R_X}_{\fS_{\delta,d+1}^{\vee}}/\abs{R_X(\Xi)})=[\widehat{\Div}(\underline{\xi^0})\cdot\dots\cdot \widehat{\Div}(\underline{\xi^d})\cdot X^{\an}]_{(K,\abs{\ndot})}\]

The Chow metric can be calculated by dimension induction: write $X_1$ the intersection of $X$ by the hyperplane $\xi^0$ and $\Xi_{1}:=(\underline{\xi}^1,\dots,\underline{\xi}^d)\in \check{E}^{d}$, it holds that (counted with multiplicities) (see \cite[Formula (3.17)]{CM})
\[\log~(\norm{R_{X_1}}_{\fS_{\delta,d}^{\vee}}/\abs{R_{X_1}(\Xi_{1})})-\log~(\norm{R_X}_{\fS_{\delta,d+1}^{\vee}}/\abs{R_X(\Xi)})=\int_{X^{\an}}\log~\abs{\underline{\xi}^0}_{\phi_{\norm{\ndot}}}c_1(\phi_{\norm{\ndot}})^{d+1} \]


These constructions are compatible with respect to valued base field change $(K',\abs{\ndot}')/(K,\abs{\ndot})$.
\subsubsection{Intersection line bundle and Monge-Ampère measure}(see \cite{CL}\cite{BE})~

The intersection line bundle $\braket{O(1)|_X}^{\braket{d+1}}$ over $\spec K$ is determined by the image of $(d+1)$-fold product of $c_1(O(1)|_X)$ under the map $\Pic(X)^{d+1}\to\Pic(\spec K)$. Write $X_1$ the intersection of $X$ by the hyperplane $\xi^0$, the restriction induces an isomorphism $\braket{O(1)|_X}^{\braket{d+1}}\simeq \braket{O(1)|_{X_1}}^{\braket{d}}$.

The metric $\phi_{\norm{\ndot}}$ induces a Monge-Ampère measure denoted by $c_1(\phi_{\norm{\ndot}})^{\wedge d}$ on $X^{\an}$, it is supported on the finite set of Shilov points on $X^{\an}$ with respect to $(L,\phi_{\norm{\ndot}})$. A metric $\braket{\phi_{\norm{\ndot}}}^{\braket{d+1}}$ on the intersection line bundle can be determined by dimension induction: with respect to the restriction isomorphism, it holds that (counted with multiplicities)
\[\log\{\braket{\underline{\xi^0},\dots,\underline{\xi^d}}_{\braket{\phi_{\norm{\ndot}}}^{\braket{d+1}}}/\braket{\underline{\xi^1}|_{X_1},\dots,\underline{\xi^d}|_{X_1}}_{\braket{\phi_{\norm{\ndot}}|_{X_1}}^{\braket{d}}}\}=\int_{X^{\an}}\log\abs{\underline{\xi^0}}_{\phi_{\norm{\ndot}}}c_1(\phi_{\norm{\ndot}})^{\wedge d}.\]
The resulting metric depends on the norm $\norm{\ndot}$ in a multi-linear way and is continuous in each argument \cite[Theorem 8.16]{BE}. 

The intersection line bundle and the induced metric on it turn out to be given by the Chow form and Chow norm, as demonstrated in \cite[Theorem 1.4]{Zha1}\cite[Lemma 1.1, Proposition 1.2]{Bos1} (see also \cite{Sou}\cite{Phi})
\begin{proposition}\label{P: Deligne-Chow isometry}
There is an isomorphism of line bundles 
\[\braket{O(1)|_X}^{\braket{d+1}}\simeq \{R_X\}^{\vee}:\quad\braket{\Xi}\mapsto (R_X\mapsto R_X(\Xi))\]
This isomorphism is an isometry: for any $\Xi$ such that $R_X(\Xi)\neq 0$, it holds that
\[\log~\braket{\Xi}_{\braket{\phi_{\norm{\ndot}}}^{\braket{d+1}}}=\log~(\abs{R_X(\Xi)}/\norm{R_X}_{\fS_{\delta,d+1}^{\vee}}).\]
\end{proposition}
\begin{proof}
Both sides satisfy the same induction by dimension formula.
\end{proof}

\subsection{G.I.T.~stability and critical metric}
\subsubsection{General reminders about action and norm}(see \cite[Chap.~5]{Ber}\cite{RTW})~

Let $V$ be a vector space over $(K,\abs{\ndot})$ with a strict Cartesian norm $\norm{\ndot}$. Let $\bbG$ be a reductive algebraic group over $K$ acting linearly on $V$. The subset of $\bbG^{\an}$ whose action preserves the norm $\norm{\ndot}$ is an analytic subgroup, denoted by $\bbU_{\norm{\ndot}}$; it is maximal compact, and is a strict affinoid subset, equivalently it is induced by an integral model group $\KG$ over $K^{\circ}$. The corresponding residual group $\widetilde{\bbG}$ is the reduction of the affinoid group $U_{\norm{\ndot}}$ which is a reductive algebraic group over $\widetilde{K}$; the action of $\bbG$ on $V$ together with the norm $\norm{\ndot}$ induces the \emph{residual action} of $\widetilde{\bbG}$ on $\widetilde{V}$.

A vector $v\in V\setminus\{0\}$ is called \emph{semistable} if the Zariski closure of its $\bbG$-orbit does not contain $0$, \emph{(norm) minimal (with respect to $\bbU_{\norm{\ndot}}$)} if $\norm{v}=\inf_{g\in \bbU_{\norm{\ndot}}}\norm{g.v}$, and \emph{residually semistable} if, after rescaling so that $\norm{v}=1$, the reduction $\widetilde{v}\in\widetilde{V}$ is semistable for the residual action by $\widetilde{\bbG}$. 

As $\overline{K}=K$, any rank one torus $\ft$ of $\bbG$ is necessarily $\bbG_m$ given by $\spec K[T^{\pm 1}]$. A torus $\bbT$ of $\bbG$ is called \emph{compatible with $\norm{\ndot}$} if it is simultaneously diagonalizable with respect to an orthonormal basis of $\norm{\ndot}$. This condition implies that $\bbT$ comes from a torus of the integral model $\KG$.

\subsubsection{Kempf-Ness criteria over non-Archimedean fields}(see \cite{Bur}\cite{Mac})~

The non-Archimedean version of the Kempf-Ness criteria is a metric approach to GIT stability\footnote{Established by Burnol for $p$-adic fields and generalized by Maculan to non-Archimedean fields.}; it shall be used in computations of the criticality condition in this and the next section.
\begin{criteria}\label{C: KN criteria}
Assume that\footnote{Only here $K$ is not assumed to be algebraically closed.} $(K,\abs{\ndot})$ is discrete and universally Japanese. The following are equivalent for a non-zero vector $v$ of $(V,\norm{\ndot})$ under a linear action of $\bbG$: (1) it is norm minimal with respect to $\bbU_{\norm{\ndot}}$; (1)'(possibly after passing to $\overline{K}$) it is norm minimal with respect to any rank one torus $\ft$ (or with respect to any maximal torus $\bbT$) of $\bbG$ that is compatible with $\norm{\ndot}$; (2) it is residually semistable; (3) there exists an invariant homogeneous polynomial $f$ in $(\fS_{K^{\circ}}V^{\circ})^{\bbU_{\norm{\ndot}}}$ such that $\abs{f(v)}=\norm{v}^{\frac{1}{\deg f}}$. \cite[Proposition 1]{Bur}\cite[Theorem 4.9]{Mac}
\end{criteria}
\begin{remark}
Note that the two assumptions on the base field are required so as to apply results in \cite[Theorem 2]{Ses} to deduce the finite generation for the corresponding invariant algebra over $K^{\circ}$. If the fininteness can be obtained from other ways, then these assumptions can be dropped; in particular, the criteria applies in our situation with algebraically closed $K$ and action by $\bbS\bbL(n+1)$ or $\bbG\bbL(n+1)$ on tensor powers of $(E,\norm{\ndot})$ since their algebras of invariants over $K^\circ$ are known to be finitely generated.
\end{remark}

\subsubsection{Chow norm criticality}(\cite{Zha1}\cite{Bos2})~

Consider the action of $\bbS\bbL(n+1)$ on $E$ and the induced action on $(\fS^{\delta} \check{E})^{\otimes d+1}$. Equip $(\fS^{\delta} \check{E})^{\otimes d+1}$ with the norm $\norm{\ndot}_{\fS_{\delta,d+1}^{\vee}}$ induced by $\norm{\ndot}$. The above action induces an action on cycles of $\bbP^n$, compatible with the formation of Chow forms/divisors in the sense that $R_{g.X}=g.R_{X}$ for any $g\in \bbS\bbL(n+1)(K)$. Write $\mu_X(g)$ the function $\bbS\bbL(n+1)(K) \ni g\mapsto \norm{g.R_{X}}_{\fS_{\delta,d+1}^{\vee}}$, as $\bbS\bbL(n+1)$ is an affine variety, by valued base field change $(K',\abs{\ndot})/(K,\abs{\ndot})$ this extends to a continuous function on $\bbS\bbL(n+1)^{\an}$. Since for any $\Xi$ and any $R_X$ with $R_X(\Xi)\neq 0$, the regular function $g\mapsto g.R_X(\Xi)/R_X(\Xi)$ is a character of $\bbG\bbL(n+1)$, hence is a multiple of $\det (g)$, thus the action by $\bbS\bbL(n+1)$ preserves $\abs{R_X(\Xi)}$.

For any rank one torus $\ft$ of $\bbS\bbL(n+1)$, the restriction of $\mu_X(\cdot)$ to $\ft^{\an}$ descends to a function $\mu_{X,\ft}(\cdot)$ on $\bbR$ via the map $\bbG_m^{\an}\to\bbR$ by $g\mapsto \log~\abs{z(g)}$; write $\gamma$ a section. If the action of $\ft$ is compatible with $\norm{\ndot}$ with a diagonalizing orthonormal basis $\underline{e}$, there are weights $\underline{\lambda}\in\underline{\bbR}$ satisfying $\sum_{i=0}^n\lambda_i=0 $ such that $\norm{\gamma_t.e_i}=\mathrm{e}^{\lambda_it}$ for any section, and write $\norm{\ndot}_t$ the unique Cartesian norm on $E$ admitting $\underline{e}$ as orthogonal basis with $\norm{e_i}_t=\mathrm{e}^{\lambda_it}$, its dual norm is the one admitting $\underline{e}^{\vee}$ as orthogonal basis with $\norm{e_i^{\vee}}^{\vee}_t=\mathrm{e}^{-\lambda_it}$,

The algebraic group $\bbS\bbL(n+1)$ acts on $E$, hence its $K$-points acts on the space of strict diagonalizable norms on $E$: if $\underline{e}$ is an orthonormal basis for $\norm{\ndot}$, then $g^*\norm{\ndot}$ is the norm with orthonormal basis $g.\underline{e}$; consequently it induces a norm $g^*\norm{\ndot}_{\fS_{\delta,d+1}^{\vee}}$ on $(\fS^{\delta} \check{E})^{\otimes d+1}$ and a metric $\braket{\phi_{g^*\norm{\ndot}}}^{\braket{d+1}}$. Actually, by construction $g^*\norm{R_X}_{\fS_{\delta,d+1}^{\vee}}$ is equal to $\norm{g.R_X}_{\fS_{\delta,d+1}^{\vee}}$ and to $\norm{R_{g.X}}_{\fS_{\delta,d+1}^{\vee}}$, hence also to $\braket{\Xi}_{\braket{\phi_{g^*\norm{\ndot}}}^{\braket{d+1}}}$. This equivariance is compatible with any valued base field extension. In particular, for a compatible rank one torus $\ft$, it holds that $\mu_{X,\ft}(t)$ is equal to $\norm{\gamma_t.R_X}_{\fS_{\delta,d+1}^{\vee}}$ and to $\braket{\phi_{\norm{\ndot}_t}}^{\braket{d+1}}$.

\begin{definition}
A strict diagonalizable Fubini-Study metric $\phi_{\norm{\ndot}}$ on $O(1)$ is \emph{critical} if the function $\mu_X(g)$ achieves minimum at the identity element of $\bbS\bbL(n+1)^{\an}$.
\end{definition}
\begin{remark}
By Proposition \ref{P: Deligne-Chow isometry}, this condition is equivalent to the minimality of the metric $\braket{\phi_{\norm{\ndot}}}^{\braket{d+1}}$ under the $\bbS\bbL(n+1)$ action, as the action preserves $\abs{R_X(\Xi)}$. Moreover, it suffices to test the minimality at $t=0$ of $\mu_{X,\ft}(t)$ for any rank one torus $\ft$ of $\bbS\bbL(n+1)$ whose action is compatible with $\norm{\ndot}$. 
\end{remark}

\subsection{Critical Fubini-Study metric on ample line bundle}~

Write $\pmb{1}$ the diagonal vector $\sum_{i\in\Card{n}}\bfe_i$ of $\bbR^{n+1}$, any vector $\lambda\in \bbR^{n+1}$ with $ \sum_{i=0}^n\lambda_i=0$ defines a linear function $\lambda(\ndot)$ on $\bbR^{n+1}/\pmb{1}\bbR$; denote by $\pmb{0}$ the origin of the quotient space. For any $I\subseteq\Card{n}$, let $\triangle_{I}$ be the simplex in $\bbR^{n+1}/\pmb{1}\bbR$ spanned by (images of) $\{e_i\}_{i\in I}$. Write $\sum$ for the Minkowski sum of polytopes and $\mu\cdot$ for the scaling of a polytope by $\mu\in\bbR_{+}$. 

Write the Monge-Ampère measure $c_1(\phi_{\norm{\cdot}})^{\wedge d}$ as $\sum_{a\in\Sh}\mu_a\delta_{\gamma_a}$, with $\Sh$ the finite set of index for Shilov points $\{\gamma_a\}$ on $X^{\an}$. For any $\gamma_a$ viewed as a point in $\bbP(E)^{\an}$, write $I_{\underline{e}}(\gamma_a)$ simply as $I_{\underline{e}}(a)$, namely the subset of $\Card{n}$ such that $\abs{e_i(\gamma_a)}_{\phi_{\norm{\cdot}}}=1$ (i.e.~the function $\abs{e_i(\cdot)}_{\phi_{\norm{\cdot}}}$ achieves its maximum $1$ at $z=\gamma_a$). 
\begin{definition}
With respect to the orthonomal basis $\underline{e}$, the \emph{Bergman distorsion} $b_{\underline{e}}(x)[\ndot]$ is the functional $\lambda\mapsto\max_{i\in I_{\underline{e}}(x)}\{-\lambda_i\}$ for any $x\in\bbP(E)^{\an}$; the \emph{Monge-Ampère polytope} $P_{MA}(X,\norm{\cdot},\underline{e})$ is the polytope $\sum_{a\in\Sh}\mu_a \triangle_{I(a)}$ in $\bbR^{n+1}/\pmb{1}\bbR$.
\end{definition}

\begin{lemma}\label{L: variation of FS metric}
Let $\ft$ be a rank one torus compatible with $\norm{\ndot}$, with a diagonalizing orthonormal basis $\underline{e}$ and corresponding action weight $\lambda$. Then for $x\in(\bbP^n)^{\an}$ and $t\in\bbR$ it holds that
\[\phi_{\norm{\ndot}_t}(x)-\phi_{\norm{\ndot}_0}(x)=-\log~(\norm{\hat{x}}^{\vee}_0/\norm{\hat{x}}^{\vee}_t)=t\cdot b_{\underline{e}}(x)[\lambda].\]
\end{lemma}
\begin{proof}
Choose an affine lift $\hat{x}$ for $x$ and any $s\in E$ with $s(x)\neq 0$, one has
\[\phi_{\norm{\ndot}^{\vee}_t}(x)-\phi_{\norm{\ndot}^{\vee}_0}(x)=\log(\abs{s(\hat{x})}/\norm{\hat{x}}^{\vee}_t)-\log(\abs{s(\hat{x})}/\norm{\hat{x}}^{\vee}_0)=-\log~(\norm{\hat{x}_t}^{\vee}/\norm{\hat{x}}^{\vee}_0).\]
Write $\hat{x}=\sum a_ie_i^{\vee}\in E^{\vee}_{K'}$, the choice of orthonormal coordinates implies that
\begin{equation*}
    \begin{split}
        &\log~(\norm{\hat{x}}^{\vee}_t/\norm{\hat{x}}^{\vee}_0)=\log~(\max_{0\leq i\leq n}\{\abs{a_i}\mathrm{e}^{-\lambda_i t}\}/ \max_{0\leq i\leq n}\{\abs{a_i}\})\\
        =&t\cdot \max_{i\in I_{\underline{e}}(\hat{x})}\{-\lambda_i\}=t\cdot  b_{\underline{e}}(x)[\lambda].
    \end{split}
\end{equation*}

\end{proof}

\begin{theorem}\label{T: critical metric}
Assume that $(X,O(1))$ is Chow stable. A diagonalizable Fubini-Study metric $\phi_{\norm{\ndot}}$ is critical if and only if $\pmb{0}\in P_{MA}(X,\norm{\cdot})$.
\end{theorem}

\begin{proof}
Let $\ft$ be a rank one torus of $\bbS\bbL(n+1)$ compatible with $\norm{\ndot}$, choose an orthonormal basis $\{e_i\}$ for $\norm{\cdot}$ that diagonalizes $\ft$, by the continuity of metric on the intersection line bundle in each of its argument, the variation of Chow height is (write $\phi^{((k))}$ for the $k$-fold repetition of $\phi$ in $\braket{\dots}$)
\begin{equation*}
    \begin{split}
        &\braket{\phi_{\norm{\cdot}_t}}^{\braket{d+1}}-\braket{\phi_{\norm{\cdot}_0}}^{\braket{d+1}}\\
        =&\sum_{k=0}^{d}\braket{\phi_{\norm{\cdot}_t}-\phi_{\norm{\cdot}_0},\phi_{\norm{\cdot}_t}^{((d-k))},\phi_{\norm{\cdot}_0}^{((k))}}\\
        =&(d+1)\braket{\phi_{\norm{\cdot}_t}-\phi_{\norm{\cdot}_0},\phi_{\norm{\cdot}_0}^{((d))}}+o(t^2)\\
        =&t\cdot (d+1)\braket{b_{\underline{e}}(x)[\lambda], \phi_{\norm{\cdot}_0}^{((d))}}+o(t^2)
    \end{split}
\end{equation*}

By properties of intersection pairing metric and of Monge-Ampère measure,
\begin{equation*}
    \begin{split}
        \frac{d}{dt}|_{0^+}\braket{\phi_{\norm{\cdot}_t}}^{\braket{d+1}}&=(d+1)\int_{X^{\an}}b_{\underline{e}}(x)[\lambda]\cdot c_1(\phi_{\norm{\cdot}_0})^{\wedge d}\\
        &=(d+1)\sum_{a\in \Sh}\mu_a\cdot b_{\underline{e}}(\gamma_a)[\lambda].
    \end{split}
\end{equation*}

By convex analysis, the last term can be expressed as
\[\sum_{a\in \Sh}\mu_a \cdot b_{\underline{e}}(\gamma_a)[\lambda]=\sum_{a\in\Sh}\max_{\mu_a\Delta_{I(a)}}\{-\lambda(\ndot)\}=\max_{\sum_{a\in\Sh}\mu_a\Delta_{I(a)}}\{-\underline{\lambda}(\ndot)\}=\max_{P_{\MA}(\norm{\cdot})}\{-\lambda(\ndot)\}.\]
That $\phi_{\norm{\ndot}}$ is critical, namely the above variation being non-negative for all $\lambda\in\bbR^{n+1}$ with $\sum\lambda_i=0$, is equivalent to the containment $\pmb{0}\in P_{MA}(X,\norm{\cdot},\underline{e})$.
\end{proof}
\begin{remark}
These computations are analogues of \cite[Theorem 3.4]{Zha1}.
\end{remark}

\section{Tropical aspects}

\subsection{Tropical cycles, $\delta$-forms and intersections}(see \cite{CLD}\cite{GK}\cite{Mih})~

We shall make use of tropical forms and currents on $\bbR$-affine spaces, more precisely the calculus of $\delta$-forms; in this article only the constant-weighted polyhedral complexes and piecewise-linear functions appear, as all our metric objects actually come from integral models. 

A tropical cycle in $\bbR^n$ is a locally finite polyhedral complex $\KC$ (of pure dimension $d$) in $\bbR^n$ with a collection of (constant-)weights $\{\mu_\sigma\}_{\sigma\in\KC_{(d)}}$ satisfying the balancing condition at each facet. There is an intersection product $\cdot$ for tropical cycles using the stable tropical intersection which can be calculated using the fan displacement rule. Pull-back satisfies the graph construction and direct image satisfies projection formula: for an affine map $F:\bbR^m\to \bbR^n$ and tropical cycles $C$ in $\bbR^m$ and $D$ in $\bbR^n$, write $\pi_i$ the projections from $\bbR^m\times\bbR^n$, it holds that
\[F^*(C)=(\pi_{1})_*\{(\pi_2)^*(C)\cdot \bfGamma_{F}\},\quad F_*\{F^*C\cdot D\}=F_*(D)\cdot C. \]

A $\delta$-form $B$ is a combination $\sum_{\sigma\in\KC} \alpha_{\sigma}\wedge\delta_{[\sigma,\mu_{\sigma}]}$ for a collection of (Lagerberg) (super-)forms $\{\alpha_{\sigma}\}_{\sigma\in\KC}$ on weighted polyhedra $\{[\sigma,\mu_{\sigma}]\}_{\sigma\in\KC}$ that satisfy the balancing condition at each facet. Viewed as a (super-)current $[B]$ by integration, they can be characterized by the condition that both $[B]$ and its differentials $d'[B]$ and $d''[B]$ are polyhedral \cite[Theorem 1.1]{Mih}.
Pull-back and push-forward operations exists. There is a graded-commutative $\wedge$-product on $\delta$-forms that extends the $\wedge$-product of smooth forms and the stable intersection product of tropical cycles. Pull-back satisfies the graph construction and push-forwards satisfies the projection formula: for an affine map $F:\bbR^m\to \bbR^n$ and $\delta$-forms $S$ in $\bbR^n$ and $T$ in $\bbR^m$
\[F^*(S)=(\pi_{1})_*\{(\pi_2)^*(S)\wedge \delta_{\bfGamma_{F}}\},\quad S\wedge F_*(T)=F_*\{F^*(S)\wedge T\}.\]
There are differential operators $d$ with components $d'$ and $d''$ on $\delta$-forms viewed as (super-)currents acting by integration on (super-)forms. In particular, for a piecewise linear function $\bff$ on $\bbR^n$, denote by $[\bff]$ the $(n,n)$-current $\bff\wedge[\bbR^n,\mu_{\mathrm{st}}]$, its tropical divisor $\Div(\bff)$ as a $(n-1,n-1)$-current is given by $\dd [\bff]$, also viewed as a $(1,1)$ $\delta$-form. 

\subsection{Tropicalization of the incidence divisor}~

Write $N$ the lattice $\bbZ^{n+1}/\pmb{1}$ and $M$ its dual lattice $\pmb{1}^{\perp}$ as a sub-lattice of $\check{\bbZ}^{n+1}$. The associated $\bbR$-vector spaces $N_{\bbR}$ and $M_{\bbR}$ are $\bbR^{n+1}/\pmb{1}$ and $(\bbR^{n+1}/\pmb{1})^{\vee}$. 

Choose homogeneous coordinates $\{x_i\}$ and denote by $\bbT$ the rank $n$ split algebraic torus in $\bbP^n$ defined by the condition $\prod_{i=0}^n x_i\neq 0$; write $\check{\bbT}$ the dual torus in $\check{\bbP}^n$; they are toric varieties associated with lattices $M$ and $N$. Tropicalization gives a continuous map $\bbT^{\an}\to N_{\bbR}$ by associating to $z$ the dual character $ (m\mapsto -\log~ \abs{\underline{x}^m}_z)$, concretely the element $(-\log~\abs{\underline{\hat{z}}})/\pmb{1}$ for any affine representative $\hat{z}\in E^{\vee}_{K'}$; similar construction holds for $\check{\bbT}^{\an}\to M_{\bbR}$; denote these maps by $\trop$.

Write $\iota:\bbR^{n+1}\simeq \check{\bbR}^{n+1}$ the identification map (and its inverse) assigning to $a$ the dual element defined by $a'\mapsto\sum a_i\cdot a_i'$, and by abuse of notation, also denote by $\iota$ the induced map by the composition $M_{\bbR}\to \check{\bbR}^{n+1}\xrightarrow{\iota}\bbR^{n+1}\to N_{\bbR}$ (and its inverse). Explicitly in coordinates, $\iota$ identifies the element $\lambda\in M_{\bbR}$ satisfying $\sum\lambda_i=0$ with the element $\lambda/\pmb{1}\in N_{\bbR}$. Write $\bfu$ the function $c\mapsto\min_{i\in\Card{n}}\{c_i\}$ on $\check{\bbR}^{n+1}$ and by abuse of notation the induced function on $M_{\bbR}$ via restriction.

Denote by $\bfX$ the tropicalization of (the restriction to $\bbT$) of $X$,  it is a tropical cycle in $N_{\bbR}$; denote by $\bfX_{\iota}$ the tropical cycle in $M_{\bbR}$ via the map $\iota$. Denote by $\bfH$ the tropical divisor of $\bfu$, namely the standard tropical hyperplane in $M_{\bbR}$. Write $\bfH^{\cdot l}$ its $l$-fold self-intersection, it is indeed the standard sub-fan of $\bfH$ with constant weights of codimension $l+1$ in $M_{\bbR}$. Let $\bfI$ be the tropicalization of the (restriction to $\bbT\times \check{\bbT}$) incidence divisor $\bbI$ in $N_{\bbR}\times M_{\bbR}$, and write $\bfI_{\iota}$ the tropical divisor $(\iota,\Id)^*\bfI$ in $M_{\bbR}\times M_{\bbR}$.

Note that as a $\delta$-form of bidegree $(1,1)$, $\delta_{\bfH}$ is just $\dd \bfu$; moreover, $\delta_{\bfX_{\iota}}\wedge (\dd \bfu)^{\wedge d}$ is the $\delta$-form corresponding to the tropical $0$-cycle $\bfX_{\iota}\cdot \bfH^{\cdot d}$. The Monge-Ampère measure $c_1(\phi_{\norm{\ndot}})^{\wedge d}$ on $X^{\an}$ can be defined as a tropical current by the theory of differential forms and currents on Berkovich spaces. It can be presented on the tropicalization $N_{\bbR}$, hence also on $M_{\bbR}$ via $\iota$, as $\sum_{a\in\Sh_{\underline{e}}}\mu_a\delta_{\trop(\gamma_a)}$.

Take copies of $M_{\bbR}$ (lower indexed by $\{1,2,3\}$), for $A\subseteq B$ two subsets of $\{1,2,3\}$ write $\bfp^B_A$ the projection map $\prod_{i\in B}M_{\bbR,i}\to \prod_{i\in A}M_{\bbR,i}$ (and write just $p_A$ if $B$ is the full set). Consider linear maps $(+)_{12}$ the addition $M_{\bbR,1}\times M_{\bbR,2}\to M_{\bbR,3}$ and $(-)_{13}$ the subtraction $M_{\bbR,1}\times M_{\bbR,3}\to M_{\bbR,2}$. Denote by $\bfGamma$ the graph of a linear map, then $\bfGamma_{(+)_{12}}$ coincides with $\bfGamma_{(-)_{13}}$ in $M_{\bbR,1}\times M_{\bbR,2}\times M_{\bbR,3}$. 

The push-forwards induces convolutions for two $\delta$-forms on $M_{\bbR}$, denoted by $\boxplus$ and $\boxminus$; on integration $\delta$-forms over tropical cycles, these convolutions are given by the stable Minkowski sum and difference for two tropical cycles: for $\alpha$ and $\beta$ two $\delta$-forms and $A$, $B$ two tropical cycles in $M_{\bbR}$, it holds that
\[\alpha\boxplus\beta:=((+)_{12})_*(\bfp^{12}_1\alpha\wedge \bfp^{12}_2\beta),\quad \delta_A\boxplus \delta_B=\delta_{A\boxplus B}.\]
Similarly $\boxminus$ can be defined. Denote by $\bfD$ the diagonal of $M_{\bbR}\times M_{\bbR}$, it holds that $(-)^*\delta_{\pmb{0}}=\delta_{\bfD}$.

\begin{lemma}
The following equality holds
\[\bfI_{\iota}=((+)_{12})^*\dd \bfu=((+)_{12})^*\bfH.\]
\[(\bfp_1^{12})^*\bfX_{\iota}\cdot \bfI_{\iota}=(\bfp_{12})_*\{\bfp_1^*\bfX_{\iota}\cdot \bfp_3^*\bfH\cdot \bfGamma_{(-)_{13}}\}\]
\end{lemma}
\begin{proof}
The equation for the incidence divisor $\bbI$ in homogeneous coordinates is $(\underline{x},\underline{\xi})\mapsto \braket{\underline{x},\underline{\xi}}$ on $\check{\bbR}^{n+1}\times \bbR^{n+1}$, so via the identification map, the tropicalization of this defining equation is the piecewise linear function $(\underline{a},\underline{a}')\mapsto \min_{i\in\Card{n}}\{a_i+a'_i\}$ on $\check{\bbR}^{n+1}\times \check{\bbR}^{n+1}$, hence $\bfI_{\iota}$ is the divisor of the function $(+)_{12}^*u$ restricted to $M_{\bbR}\times M_{\bbR}$. The second equality follows from the graph formula for pull-back and the projection formula for intersection of tropical cycles, together with the identification of graphs $\Gamma_{(+)_{12}}$ with $\Gamma_{(-)_{13}}$. 
\end{proof}

\subsection{Tropical Chow hypersurface and refined Chow polytope}~

Take the $(d+1)$-fold product $(M_{\bbR})^{d+1}$ and let $\bfDelta$ be the diagonal inclusion of $M_{\bbR}$ into it. By abuse of notation, as above, write $\iota$ the identification map with $(N_{\bbR})^{d+1}$; write $\bfp^T_S$ the projection maps for products of $(M_{\bbR})^{d+1}$ indexed by $S\subset T\subset\{1,2,3\}$ (and write just $\bfp_S$ if $T$ is the full set). Denote by $(+)_{\mathbf{12}}$ and by $(-)_{\mathbf{13}}$ the linear maps, by $\boxplus$ and $\boxminus$ the induced convolutions (resp.~Minkowski sum/difference) on $\delta$-forms (resp.~tropical cycles) on $(M_{\bbR})^{d+1}$. 

Denote by $\bfCh(\bfX)$ the tropicalization of (the restriction to $\check{\bbT}^{d+1}$) the Chow hypersurface $\Ch(X)$, it is a tropical hypersurface in $(M_{\bbR})^{d+1}$, write $\bfCh(\bfX)_{\iota}$ its identification in $(N_{\bbR})^{d+1}$. The pull-back $\bfDelta^*\bfCh(\bfX)$ is a tropical hypersurface in $M_{\bbR}$. We call both of them the tropical Chow hypersurface of the tropical cycle $\bfX$. 

\begin{proposition}\label{P: trop chow hypersurface}
The tropical Chow hypersurfaces are given by
\[\bfCh(\bfX)=(\bfp^{12}_2)_*\{(\bfp^{12}_1)^*\bfDelta_*\bfX_{\iota}\cdot \prod_{i=0}^d\bfI_{i,\iota}\}=\bfDelta_*(\bfX_{\iota})\boxminus\prod_{i=0}^d\bfH_i,\]
and
\[\bfDelta^*\bfCh(\bfX) =\bfDelta^*(\bfp^{12}_1)_*[(\bfp^{12}_1)^*\bfDelta_*\bfX_{\iota}\cdot \prod_{i=0}^d\bfI_{i,\iota}]=\bfX_{\iota}\boxminus \bfH^{\cdot(d+1)}.\]
\end{proposition}
\begin{proof}
First calculate the incidence intersection using graph formula for pull-back:

\begin{equation*}
    \begin{split}
        (\bfp^{12}_1)^*\bfDelta_*\bfX_{\iota} \cdot \prod_{i=0}^d\bfI_{i,\iota}&=(\bfp^{12}_1)^*\bfDelta_*\bfX_{\iota} \cdot (+)_{\mathbf{12}}^*\prod_{i=1}^d \bfH_i\\
        &=(\bfp^{12}_1)^*\bfDelta_*\bfX_{\iota} \cdot (\bfp_{12})_*\{(\bfp_3)^* \prod_{i=1}^d \bfH_i\cdot\mathbf{\Gamma}_{(+)_{\mathbf{12}}}\}\\
        &=(\bfp_{12})_*\{(\bfp_{12})^*(\bfp^{12}_1)^*\bfDelta_*\bfX_{\iota}\cdot (\bfp_3)^* \prod_{i=1}^d \bfH_i\cdot\mathbf{\Gamma}_{(+)_{\mathbf{12}}}\}\\
        &=(\bfp_{12})_*\{(\bfp_1)^*\bfDelta_*\bfX_{\iota}\cdot (\bfp_3)^* \prod_{i=1}^d \bfH_i\cdot\mathbf{\Gamma}_{(-)_{\mathbf{13}}}\}
    \end{split}
\end{equation*}

Then calculate the correspondence transform using projection formula for push-forward:
\begin{equation*}
    \begin{split}
        \bfCh(\bfX)&=(\bfp^{12}_2)_*\{(\bfp^{12}_1)^*\bfDelta_*\bfX_{\iota}\cdot \prod_{i=0}^d\bfI_{i,\iota}\}\\
        &=(\bfp^{12}_2)_*\{(\bfp_{12})_*((\bfp_1)^*\bfDelta_*\bfX_{\iota}\cdot (\bfp_3)^* \prod_{i=0}^d \bfH_i\cdot\mathbf{\Gamma}_{(-)_{\mathbf{13}}})\}\\
        &=(\bfp_2)_*\{(\bfp_1)^*\bfDelta_*\bfX_{\iota}\cdot (\bfp_3)^* \prod_{i=0}^d \bfH_i\cdot\mathbf{\Gamma}_{(-)_{\mathbf{13}}}\}\\
        &=\bfDelta_*\bfX_{\iota}\boxminus\prod_{i=0}^d \bfH_i,
    \end{split}
\end{equation*}
and
\begin{equation*}
    \begin{split}
        \bfDelta^*\bfCh(\bfX)&=\bfDelta^*(\bfDelta_*\bfX_{\iota}\boxminus\prod_{i=0}^d \bfH_i)=\bfX_{\iota}\boxminus \bfH^{\cdot(d+1)}.
    \end{split}
\end{equation*}

\end{proof}
\begin{remark}
These two expressions for the tropical Chow hypersurfaces are firstly obtained in \cite{Fin} and in \cite{Tri} based on results in \cite{DFS}; we recast it into modern language of tropical intersection theory with $\delta$-form calculus.
\end{remark}
By construction, the Chow divisors are given by (unique up to an additive constant) piecewise linear function $\bfR_{\bfX}$ (resp.~$\bfDelta^*\bfR_{\bfX}$) on $M_{\bbR}^{d+1}$ (resp.~on $M_{\bbR}$), call them \emph{tropical Chow forms}.

\begin{corollary}\label{C: trop Chow forms}
The tropical Chow forms are given by convolution of $\delta$-forms
\[\bfR_{\bfX}=\delta_{\bfDelta_*\bfX_{\iota}} \boxminus (\frac{1}{d+1}\bfu_i\sum_{i=0}^d\bigwedge_{j\neq i}\delta_{\bfH_j}),\quad \bfDelta^*\bfR_{\bfX}=\delta_{\bfX_{\iota}}\boxminus \bfu \delta_{\bfH^{\cdot d}}.\]
\end{corollary}
\begin{proof}
First note that the incidence intersection has a Green $\delta$-form since
\begin{equation*}
    \begin{split}
        &\dd \{(\bfp_1)^*\delta_{\bfDelta_*\bfX_{\iota}}\wedge (\bfp_3)^*(\frac{1}{d+1}\bfu_i\sum_{i=0}^d\bigwedge_{j\neq i}\delta_{\bfH_j})\wedge \delta_{ \bfGamma_{(-)_{\mathbf{13}}}}\}\\
        =&\delta_{(\bfp_1)^*\bfDelta_*\bfX_{\iota}}\wedge \delta_{(\bfp_3)^*\prod_{i=0}^d \bfH_i} \wedge \delta_{ \bfGamma_{(-)_{\mathbf{13}}}}
        =\delta_{(\bfp_1)^*\bfDelta_*\bfX_{\iota}\cdot (\bfp_3)^*\prod_{i=0}^d \bfH_i \cdot \bfGamma_{(-)_{\mathbf{13}}} }.
    \end{split}
\end{equation*}
Thus the push forwards are Green $\delta$-functions for tropical Chow divisors, hence they give tropical Chow forms (up to an additive constant):
\begin{equation*}
    \begin{split}
        \bfR_{\bfX}&=(\bfp_2)_*\{(\bfp_1)^*\delta_{\bfDelta_*\bfX_{\iota}}\wedge (\bfp_3)^*(\frac{1}{d+1}\bfu_i\sum_{i=0}^d\bigwedge_{j\neq i}\delta_{\bfH_j}) \wedge \delta_{ \bfGamma_{(-)_{\mathbf{13}}}}\}\\
        &=\delta_{\bfDelta_*\bfX_{\iota}} \boxminus (\frac{1}{d+1}\bfu_i\sum_{i=0}^d\bigwedge_{j\neq i}\delta_{\bfH_j})
    \end{split}
\end{equation*}

and
\begin{equation*}
    \begin{split}
        \bfDelta^*\bfR_{\bfX}&=\bfDelta^*\{\delta_{\bfDelta_*\bfX_{\iota}} \boxminus (\frac{1}{d+1}\bfu_i\sum_{i=0}^d\bigwedge_{j\neq i}\delta_{\bfH_j})\}=\delta_{\bfX_{\iota}}\boxminus \bfu \delta_{\bfH^{\cdot d}}
    \end{split}
\end{equation*}
\end{proof}

\subsection{Subdifferential and Legendre-Fenchel transform}(see \cite[Chap.~2]{BPS})~

Let $\bff$ be a \emph{concave} function on $M_{\bbR}$, its \emph{Legendre-Fenchel dual}
$\bff^{\vee}$ is the concave function on $N_{\bbR}$ defined by $b\mapsto \sup_{a\in M_{\bbR}}(\braket{b,a}-\bff(a))$. 
The sup-differential of a concave function $\bff$ on $M_{\bbR}$ at a point $a$ is the subset $\partial \bff(a)$ of $N_{\bbR}$ consisting of points $b$ such that $\braket{b,(a'-a)}\geq \bff(a')-\bff(a)$ holds for all $a'\in M_{\bbR}$. For any $\tau\in M_{\bbR}$, denote by $\tau:M_{\bbR}\to M_{\bbR}$ the translation $a\mapsto a-\tau$, and by $\partial_{\tau}\bff(a)\in N_{\bbR}$ the partial derivative of $\bff$ at $a$ in direction $\tau$; namely the element determined by $\bff(a+\tau)-\bff(a)=\braket{\partial_{\tau}\bff(a),\tau}+o(\abs{\tau})$. The sup-differential $\partial \bff$ induces set-valued maps from $M_{\bbR}$ to $N_{\bbR}$.

Recall that $\bfu$ is the piecewise linear function on $M_{\bbR}$ induced by $c\mapsto\min_{i\in\Card{n}}\{c_i\}$ on $\check{\bbR}^{n+1}$, write $I(a)$ the subset of $\Card{n}$ for indices such that $a_i$ is minimal, namely $\bfu(a)=a_i$. For any $\tau\in M_{\bbR}$, write $\bfu_{\tau}$ the translated function $a\mapsto\bfu(a-\tau)$. A standard calculation gives $(\dd \bfu_{\tau})^{\wedge n}=\delta_{\tau}$. In general, for a $\delta$-form $B$ on $M_{\bbR}$, write $\partial_{\tau}B$ its partial derivative in direction $\tau$.


\begin{lemma}\label{L: standard pwl function}
The set-valued map $\partial\bfu$ can be computed as $a\mapsto \triangle_{I(a)}$; in particular, $\bfu^{\vee}$ is equal to the indicator function on the standard simplex $\triangle_{\Card{n}}$. 
\end{lemma}
\begin{proof}
Standard calculation.
\end{proof}

\begin{lemma}\label{L: corner locus indicator set}
With an orthonomal basis $\underline{e}$, for any $\underline{z}\in\bbT^{\an}$, the indicator sets satisfy $I(\trop_{\underline{e}}(\underline{z}))=I_{\underline{e}}(\underline{z})$.
\end{lemma}
\begin{proof}
Write $\underline{z}$ as corresponding to $\hat{z}\in E^{\vee}_{K'}$ with $K'/K$ suitable valued field extension. By orthonormality the indices $i\in\Card{n}$ belonging to $I_{\underline{e}}(\underline{z})$ are those such that $\abs{z_i}$ is maximal. These are indices that $-\log~\abs{z_i}$ is minimal, hence coincides with the subset $I(\trop_{\underline{e}}(\underline{z}))$.
\end{proof}

\begin{lemma}
For the $\delta$-form $\bfu\wedge\delta_{\bfH^{\cdot (d)}}$, it holds that
\[\partial_{\tau}(\bfu\wedge\delta_{\bfH^{\cdot (d)}})=(d+1)(\partial_{\tau}\bfu)\wedge\delta_{\bfH^{\cdot (d)}}.\]
\end{lemma}
\begin{proof}
By the continuity of wedge product for $\delta$-forms and the integration by part formula, one has
\begin{equation*}
    \begin{split}
        &\tau^*(\bfu(\dd \bfu )^{\wedge d})=\bfu_{\tau}(\dd \bfu_{\tau})^{\wedge d}\\
        &=(\bfu_{\tau}-\bfu)(\dd \bfu_{\tau})^{\wedge d}+\bfu\dd(\bfu_{\tau}-\bfu)\wedge\sum_{l=0}^{d-1}(\dd\bfu_{\tau})^{\wedge d-1-l}\wedge (\dd\bfu)^{l}\\
        &=(\bfu_{\tau}-\bfu)(\dd \bfu_{\tau})^{\wedge d}+(\bfu_{\tau}-\bfu)\dd \bfu\wedge\sum_{l=0}^{d-1}(\dd\bfu_{\tau})^{\wedge d-1-l}\wedge (\dd\bfu)^{l}\\
        &=\tau (d+1)\bfu(\dd\bfu)^{\wedge d}+o(|\tau|),
    \end{split}
\end{equation*}
hence the desired identity holds as $\dd \bfu=\delta_{\bfH}$.
\end{proof}

\begin{proposition}\label{convolution derivative}
Let $A$ and $B$ be $\delta$-forms of bidegree $(l,l)$ and $(n-l,n-l)$ on $M_{\bbR}$ and $\bff$ be the piecewise linear function on $M_{\bbR}$ defined by $\bff=A\boxminus B$. It holds
\[\partial_{\tau}\bff=A\boxminus (-\partial_{\tau} B),\quad \bff(\pmb{0})=\int_{M_{\bbR}}A\wedge B \]
\end{proposition}
\begin{proof}
This follows from the property of convolution
\[\tau^*\bff=\tau^*(A\boxminus B)=A\boxminus (-\tau)^*B,\]
and the projection formula combined with $(-)^*\delta_{\pmb{0}}=\delta_\bfD$.
\end{proof}

\begin{corollary}\label{C: trop Chow form LF dual}
The following equalities hold
\[\partial_{\tau} \bfDelta^*\bfR_{\bfX}=\delta_{\bfX_{\iota}}\boxminus(-\partial_{\tau} \bfu)\delta_{\bfH^{\cdot(d)}},\quad \partial \bfDelta^*\bfR_{\bfX}(\pmb{0})=(-)^*(d+1)\sum_{a\in\Sh}\lambda_a\triangle_{I(a)}.\]
\end{corollary}
\begin{proof}
Take $\bff$ to be $\bfDelta^*\bfR_{\bfX}$ and use its convolution expression in Corollary \ref{C: trop Chow forms}, for any $\tau\in M_{\bbR}$, the partial derivative $\partial_{\tau}\bff$ is calculated as a convolution
\[\partial_{\tau} \bfDelta^*\bfR_{\bfX}=\partial_{\tau} \{ \delta_{\bfX_{\iota}}\boxminus \bfu \delta_{\bfH^{\cdot (d)}}\}=\delta_{\bfX_{\iota}}\boxminus-\partial_{\tau}\{ \bfu \delta_{\bfH^{\cdot (d)}}\}=\delta_{\bfX_{\iota}}\boxminus(-\partial_{\tau} \bfu)\delta_{\bfH^{\cdot(d)}},\]
the evaluation at $\pmb{0}$ thus can be expressed as
\[\partial_{\tau} \bfDelta^*\bfR_{\bfX}(\pmb{0})=-\int_{M_{\bbR}}\delta_{\bfX}\wedge\partial_{\tau}\bfu(\dd\bfu)^{\wedge d}=-\int_{M_{\bbR}}\partial_{\tau}\bfu~ \delta_{[\bfX\cdot (\bfH)^{\cdot (d)}]}=-(d+1)\sum_{a\in\Sh}\lambda_a\partial_{\tau}\bfu(\gamma_a).\]
The equality of polytopes follows from this equality of functions.
\end{proof}

\subsection{Residual Chow stability}~

Denote by $\bbS$ a maximal torus of $\bbG\bbL(n+1)$. Choose a basis $\underline{e}$ of $E$ that diagonalizes the action of $\bbS$ on $E$, hence also diagonalizes its induced action on $\fS^l E^{\vee}$, the eigenspaces are monomials $\{\underline{\check{e}}^{I}\}$ indexed by weights $I\subseteq \bbN^{n+1}$ with $\abs{I}=l$. Identify $\bbG_m^{n+1}$ in $E$ defined by $\underline{e}$ with $\bbS$. For any homogeneous polynomial $f\in \fS^l E^{\vee}$ written as $\sum a_I \underline{z}^I$, write $P_{\wt}(f,\bbS)$ its weight polytope, the convex hull of its non-zero weights, and write $\vartheta(F,\bbS)$ the weight (or roof) function defined over this polytope with graph the upper hull of augmented weights $\{(I, \log~\abs{a_I})\}$, and write $\hat{P}_{\wt}(f,\bbS)$ the augmented weight polytope as the hypograph of the weight function $\vartheta$. 

Recall that a strict Cartesian norm $\norm{\ndot}$ on $E$ is compatible with the action by $\bbS$ if there exists an orthonormal basis that diagonalizes the action. With such a norm and its induced norm on $\fS^l E^{\vee}$, write $\vartheta(\widetilde{f},\widetilde{\bbS})$ the weight function of the residual action, the indicator function of the weight polytope $P_{\wt}(\widetilde{f},\widetilde{\bbS})$ for residual action over $\widetilde{K}$.

Write $\bff_{\underline{e}}$ the tropicalization of $f$ (restricted to $\bbS$), it is a piecewise linear function on $\check{\bbR}^{n+1}$.
\begin{lemma}\label{L: weight polytope as tropical function}
For $f\in \fS^l E^{\vee}$ with its tropicalization $\bff_{\underline{e}}$, it holds that $\vartheta(f,\bbS)=\bff_{\underline{e}}^{\vee}$; with a compatible norm $\norm{\ndot}$, it holds that $P_{\wt}(\widetilde{f},\widetilde{\bbS})=\partial\bff_{\underline{e}}(\pmb{0})$.
\end{lemma}
\begin{proof}
The first assertion is a well-known fact in tropical geometry for hypersurfaces. The second assertion follows from the observation that non-zero weights in $\widetilde{f}$ are those monomials $\underline{z}^I$ with maximal $\abs{a_I}$, hence the residual weight polytope is isomorphic to the upper horizontal face of $\hat{P}_{\wt}(f,\bbS)$, or equivalently that of the graph of $\bff_{\underline{z}}^{\vee}$, which is by definition expressed as the sup-differential of $\bff_{\underline{z}}$ at $\pmb{0}$.
\end{proof}

Consider a maximal (rank $n$) torus $\bbT$ of $\bbS\bbL(n+1)$, it is the restriction of a maximal (rank $n+1$) torus $\bbS$ of $\bbG\bbL(n+1)$. The inclusion induces an action of $\check{\bbT}^{d+1}$ on $R_X$ and an action by $\check{\bbT}$ on $\Delta^*R_X$ induced via the embedding $\Delta$. 

\begin{lemma}\label{L: weight diagonal transposition}
It hols that
\[\vartheta(R_X,\Delta\check{\bbT})=\vartheta(\Delta^*R_X, \check{\bbT}),\quad P_{\wt}(\widetilde{R_X},\widetilde{\Delta\check{\bbT}})=P_{\wt}(\widetilde{\Delta^*R_X},\widetilde{\check{\bbT}})\]
\end{lemma}
\begin{proof}
Unfold the definition of weight functions of an action.
\end{proof}



Denote by $\tau(R_X)$ the tropicalization of (the restriction to $\check{\bbS}$ of) $\Delta^*R_X$, it is a piecewise linear function on $\check{\bbR}^{n+1}$; its restriction to $M_{\bbR}$ is also a piecewise linear function. 

\begin{proposition}\label{P: weight transpose}
As functions on $M_{\bbR}$, it holds that up to an additive constant
\[\tau(R_X)|_{M_{\bbR}}=\bfDelta^*\bfR_{\bfX}. \]
As functions and polytopes on $N_{\bbR}$, it holds that (up to an additive constant for functions)
\[\vartheta(\Delta^*R_X,\check{\bbT})= (\bfDelta^*\bfR_{\bfX})^{\vee},\quad P_{\wt}(\widetilde{\Delta^*R_X},\widetilde{\check{\bbT}})=\partial\bfDelta^*\bfR_{\bfX}(\pmb{0}).\]
In particular, the refined Chow polytope is given by $\dom ((\bfDelta^*\bfR_{\bfX})^{\vee})$ with the induced lower face projection subdivision.
\end{proposition}
\begin{proof}
The first assertion comes from the commutativity of pull-back of Cartier divisor with tropicalization. The second assertion follows from Proposition \ref{L: weight polytope as tropical function}.
\end{proof}

\begin{corollary}\label{C: MA polytope}
The residual Chow polytope is equal to the (opposite of dilated) Monge-Ampère polytope
\[P_{\wt}(\widetilde{\Delta^*R_X},\widetilde{\check{\bbT}})=\partial\bfDelta^*\bfR_{\bfX}(\pmb{0})=(-)^*(d+1)P_{\MA}(X,\norm{\cdot})\]
\end{corollary}
\begin{proof}
The second equality follows from Corollary \ref{C: trop Chow form LF dual}.
\end{proof}



\begin{theorem}\label{T': critical metric}
Assume that $(X,O(1))$ is Chow stable. The Fubini-Study metric $\phi_{\norm{\ndot}}$ is critical if and only if $\pmb{0}\in P_{MA}(X,\norm{\cdot})$.
\end{theorem}

\begin{proof}
By definition, $\phi_{\norm{\ndot}}$ being critical is equivalent to the Chow form $R_X$ being norm minimal with respect to the maximal torus $\bbT$ of $\bbS\bbL(n+1)$ induced by $\norm{\ndot}$. By Criteria \ref{C: KN criteria} this holds if and only if $R_X$ is residually semi-stable, or in other words $\pmb{0}$ belongs to the polytope $P_{\wt}(\widetilde{R_X},\widetilde{\Delta\check{\bbT}})$, hence to $P_{\wt}(\widetilde{\Delta^*R_X},\widetilde{\check{\bbT}})$ by Lemma \ref{L: weight diagonal transposition}. Thus we get the same criteria $\pmb{0}\in P_{\MA}(X,\norm{\ndot})$ by the above Proposition.
\end{proof}

\begin{remark}
As a Minkowski sum of dilated simplices, the Monge-Ampère polytope $\sum_{a\in\Sh}\mu_a \triangle_{I(a)}$ (now with simplices viewed in $\bbR^{n+1}$) is a polymatroid (or generalized permutohedra) in $\bbR^{n+1}$\cite{PRW}. The corresponding rank function on subsets of $\Card{n}$ is given by 
\[I\subseteq \Card{n},\quad r(I)=\sum_{a\in\Sh, I(a)\cap I\neq \emptyset}\mu_a.\]
It remains to study combinatorial informations encoded by this polymatroid.
\end{remark}

\bibliographystyle{abbrv}

\begin{thebibliography}{AB}


\bibitem[Ber]{Ber}
V.~Berkovich
\newblock{\em Spectral theory and analytic geometry over non-Archimedean fields}
\newblock Mathematical Surveys and Monographs (33), AMS, 1990.


\bibitem[BE]{BE}
S.~Boucksom and D.~Eriksson.
\newblock Spaces of norms, determinant of cohomology and Fekete points in non-Archimedean geometry
\newblock {\em Adv. Math. 378 (2021), Paper No. 107501}.
 
\bibitem[BGM]{BGM}
S.~Boucksom, W.~Gubler and F.~Martin
\newblock Differentiability of relative volumes over an arbitrary non-Archimedean field
\newblock {\em Int. Math. Res. Not. IMRN 2022, no. 8, 6214–6242}.

\bibitem[BGR]{BGR}
S.~Bosch, U.~Günter and R.~Remmert
\newblock Non-Archimedean Analysis, A Systematic Approach to Rigid Analytic Geometry
\newblock {\em Grundlehren der mathematischen Wissenschaften} 1984

\bibitem[BoGS]{BoGS}
J.-B.~Bost, H.~Gillet and C.~Soulé
\newblock Heights of projective varieties and positive Green forms
\newblock {\em J. AMS. 4 (1994), no. 7/4}

\bibitem[BlGS]{BlGS}
S.~Bloch, H.~Gillet and C.~Soulé
\newblock Non-Archimedean Arakelov theory
\newblock {\em J. Algebraic Geom. 4 (1995), no. 3}

\bibitem[BGJKM]{BGJKM}
J.~Burgos Gil, W.~Gubler, Ph.~Jell, K.~Künnemann and F.~Martin
\newblock Differentiability of non-archimedean volumes and non-archimedean Monge-Ampère equations
\newblock {\em Algebraic Geom. 7 (2020), no. 2}

\bibitem[BPS]{BPS}
J.~Burgos Gil, P.~Philippon and M.~Sombra
\newblock Arithmetic geometry of toric varieties. Metrics, measures and heights
\newblock {\em Astérisque no. 360 (2014)}

\bibitem[Bos1]{Bos1}
J.-B.~Bost
\newblock Semi-stability and heights of cycles
\newblock {\em Invent. Math. 118 (1994) n.2 223-253}

\bibitem[Bos2]{Bos2}
J.-B.~Bost
\newblock Intrinsic heights of stable varieties and abelian varieties
\newblock{\em Duke Math. J. 82 (1996), no. 1, 21–70}


\bibitem[Bur]{Bur}
J.-F.~Burnol
\newblock Remarques sur la stabilité en arithmétique
\newblock {\em I.M.R.N. Issue 6 (1992)}

\bibitem[Che]{Che}
H.~Chen
\newblock Maximal slope of tensor product of Hermitian vector bundles
\newblock{\em J. Alg. Geom. 18 no.3 (2009)}

\bibitem[CL]{CL}
A.~Chambert-Loir
\newblock Mesures et équidistribution sur les espaces de Berkovich
\newblock{\em J. Reine Angew. Math. 595 (2006), 215--235}

\bibitem[CLD]{CLD}
A.~Chambert-Loir and A.~Ducros
\newblock Formes différentielles réelles et courants sur les espaces de Berkovich
\newblock \href{https://arxiv.org/pdf/1204.6277.pdf}
 {\tt https://arxiv.org/pdf/1204.6277.pdf}.

\bibitem[CM]{CM}
H.~Chen and A.~Moriwaki
\newblock Arithmetic intersection theory over adelic curves
\newblock \href{https://arxiv.org/pdf/2103.15646.pdf}
 {\tt https://arxiv.org/pdf/2103.15646.pdf}.


\bibitem[Don]{Don}
S.~Donaldson
\newblock Scalar curvature and projective embeddings. I
\newblock{\em J. Differential Geom. 59 (2001), no. 3, 479–522}

\bibitem[DFS]{DFS}
A.~Dickenstein and E.~Feichtner and B.~Sturmfels
\newblock Tropical discriminants
\newblock{\em J. Amer. Math. Soc. 20 (2007), no. 4, 1111–1133}

\bibitem[Fin]{Fin}
A.~Fink
\newblock Tropical cycles and Chow polytopes
\newblock{\em Beitr. Alg. Geom., 54, 13–40 (2013)}

\bibitem[Gas]{Gas}
C.~Gasbarri
\newblock Heights and Geometric Invariant Theory
\newblock{\em Forum Math. 12 (2000) }

\bibitem[Gua]{Gua}
R.~Gualdi
\newblock Heights of hypersurfaces in toric varieties
\newblock{\em Algebra Number Theory 12 (2018), no. 10}

\bibitem[Gub]{Gub}
W.~Gubler
\newblock Local heights for subvarieties over non-archimedean fields
\newblock{\em J. Reine Angew. Math. 498 (1998), 61--113}


\bibitem[GK]{GK}
W.~Gubler and K.~Künnemann
\newblock A tropical approach to non-archimedean Arakelov theory 
\newblock {\em Algebra Number Theory 11 (2017)}

\bibitem[Luo]{Luo}
H.~Luo 
\newblock Geometric criterion for Gieseker-Mumford stability of polarized manifolds
\newblock {\em J. Diff. Geom. 49 (1998), no. 3, 577–599}

\bibitem[Mac]{Mac}
M.~Maculan
\newblock Diophantine Applications of Geometric Invariant Theory
\newblock{\em Mém. Soc. Math. Fr. (N.S.) No. 152 (2017)}

\bibitem[Mih]{Mih}
A.~Mihatsch
\newblock On Tropical Intersection Theory
\newblock \href{https://arxiv.org/pdf/2107.12067.pdf}
 {\tt https://arxiv.org/pdf/2107.12067.pdf}.

\bibitem[Phi]{Phi}
P.~Philippon
\newblock Sur les hauteurs alternatives I
\newblock{\em Math. Ann. 289 (1991) 255-283}

\bibitem[PRW]{PRW}
A.~Postnikov, V.~Reiner and L.~Williams
\newblock Faces of generalized permutohedra
\newblock{\em Documenta Mathematica Vol.13 (2008), 207-273}

\bibitem[RTW]{RTW}
B.~Rémy, A.~Thuillier and A.~Werner
\newblock Bruhat-Tits theory from Berkovich's point of view. I
\newblock{\em  Ann. Sci. Éc. Norm. Supér. (4) 43 (2010), no. 3, 461–554}


\bibitem[Ses]{Ses}
C.S.~Seshadri
\newblock Geometric reductivity over an arbitrary base
\newblock{\em Adv. Math. no. 26 (1977), 225-274}


\bibitem[Tri]{Tri}
P.~Tripoli
\newblock Tropical Chow Hypersurfaces
\newblock{\em I.M.R.N. no. 14 (2019), 4302–4324}

\bibitem[Sou]{Sou}
Ch.~Soulé
\newblock Géométrie d’Arakelov et théorie des nombres transcendants
\newblock{\em Astérisques} 198-200, 1991

\bibitem[Zha1]{Zha1}
S.-W.~Zhang
\newblock Heights and reductions of semi-stable varieties
\newblock{\em Compos. Math. (1996)}

\bibitem[Zha2]{Zha2}
S.-W.~Zhang
\newblock Geometric reductivity at Archimedean places
\newblock{\em Internat. Math. Res. Notices 1994, no. 10}


\end{thebibliography}

\end{document}